\documentclass[twoside,10.5pt]{article}
\usepackage{amsfonts}
\usepackage{mathrsfs}
\usepackage{pifont}
\usepackage{amsmath}
\usepackage{amsthm}
\usepackage{txfonts}
\usepackage{geometry}
\usepackage{latexsym}
\usepackage{amssymb}
\usepackage{graphicx}
\usepackage{geometry}
\usepackage{xcolor}

\input{tcilatex}

\begin{document}

\title{\parbox{\linewidth}{\footnotesize\noindent }Canonical $1$-form
associated with a Lie-Rinehart structures on weil bundles}
\author{Olivier MABIALA\ MIKANOU\thanks{{\footnotesize stive.elg@gmail.com}}%
, Basile Guy Richard BOSSOTO\thanks{{\footnotesize bossotob@yahoo.fr}}  \and
Marien NGOUABI University, Faculty of Science and Technology, BP : 69,
Brazzaville, Congo. }
\date{}
\maketitle

\begin{abstract}
In this paper, we denote by $\mathbb{A}$ a Weil algebra, $M$ a smooth
manifold and $M^{\mathbb{A}}$ the associated Weil bundle. We study the
properties of differential operators on $M^{\mathbb{A}}$ and construct the
canonical $1$-form when $M^{\mathbb{A}}$ is provided with a structure of
Lie-Rinehart algebra.
\end{abstract}



\section{Preliminaries}

Throughout this text, all differentiable structures are assumed to be of
class $C^{\infty }$, therefore smooth. We denote by $M$ a paracompact smooth
manifold of dimension $n$, $C^{\infty }(M)$ the algebra of functions of
class $C^{\infty }$ on $M$ and $\mathfrak{X}(M)$, the $C^{\infty }(M)$%
-module of vector fields on $M$.

\subsection{First order differential operator}

When $R$ is a commutative algebra, with unit $1_{R}$ , and when $E$ is a $R$%
-module, a linear application 
\begin{equation*}
\delta :R\longrightarrow E
\end{equation*}%
is a first order differential operator if, for all $a$ and $b$ belonging to $%
R$,%
\begin{equation*}
\delta (ab)=\delta (a)\cdot b+a\cdot \delta (b)-ab\cdot \delta (1_{R})\text{.%
}
\end{equation*}%
When $\delta (1_{R})=0$, we have the usual notion of derivation from $R\ $\
into $E$.

Thus, a linear map%
\begin{equation*}
\delta :R\longrightarrow E
\end{equation*}%
is a first order differential operator if and only if, the map%
\begin{equation*}
R\longrightarrow E,a\mapsto \delta (a)-a\cdot \delta (1_{R}),
\end{equation*}%
is a derivation.

A first order differential operator from $R$ into $R$ is simply called first
order differential operator of $R$. We denote by $\mathcal{D}(R,E)$ the $R$%
-module of first order differential operators from $R$ into $E$ and by $%
\mathcal{D}(R)$ the $R$-module of first order differential operators of $R$.

If $A$ and $B$ are arbitrary algebras and if $\varphi :A\longrightarrow B$
is a homomorphism of algebras, then $B$ is a $A$-module.

Let $\varphi :A\longrightarrow B$ be a homomorphism of algebras. A linear
map $\delta :A\longrightarrow B$ is a $\varphi $-first order differential
operator if%
\begin{equation*}
\delta (ab)=\delta (a)\cdot \varphi (b)+\varphi (a)\cdot \delta (b)-\varphi
(ab)\cdot \delta (1_{A})\text{.}
\end{equation*}

When $M$ is a smooth manifold, a first order differential operator of the
algebra $C^{\infty }(M)$ will simply said differential operator on $M$ and
denoted by $\mathcal{D}(M)$ instead of $\mathcal{D}(C^{\infty }(M))$.

\subsection{Weil algebra and Weil bundle}

A Weil algebra or local algebra (in the sense of Andr\'{e} Weil) \cite{wei},
is a finite dimensional, associative, commutative and unitary algebra $%
\mathbb{A}$ over $\mathbb{R}$ in which there exists an unique maximum ideal $%
\mathfrak{m}$\ of codimension $1$. In his case the factor space $\mathbb{A}/%
\mathfrak{m}$ is one-dimensional and is identified with the algebra of real
numbers $\mathbb{R}$. Thus $\ \mathbb{A}=\mathbb{R\oplus }\mathfrak{m}$ \
and $\mathbb{R}$ is identified with $\mathbb{R}\cdot 1_{\mathbb{A}}$, where $%
1_{\mathbb{A}}$\ is the unit of $\mathbb{A}$. The integer $k\in 
\mathbb{N}
$ such that $\mathfrak{m}^{k+1}=(0)$ and $\mathfrak{m}^{k}\neq (0)$ is
called the height of $\mathbb{A}$.

A near point of $x\in M$ of kind $\mathbb{A}$ is a homomorphism of algebras%
\begin{equation*}
\xi :C^{\infty }(M)\longrightarrow \mathbb{A}
\end{equation*}%
such that for any $f\in C^{\infty }(M),$ $[\xi (f)-f(x)]\in \mathfrak{m}$.

We denote $M_{x}^{\mathbb{A}}$ the set of near points of $x$ of kind $%
\mathbb{A}$ and $M^{\mathbb{A}}=\underset{x\in M}{\cup }M_{x}^{\mathbb{A}}$
the set of near points on $M$ of kind $\mathbb{A}$ : it is smooth manifold
of dimension $n\times \dim \mathbb{A}$. The manifold $M^{\mathbb{A}}$ is
called the manifold of infinitely near points on $M$ of kind $\mathbb{A}$ 
\cite{wei}, \cite{mor} or simply the Weil bundle\ \cite{kms}.

If $(U,\varphi )$ is a local chart of $M$ with local coordinates $%
(x_{1},x_{2},...,x_{n})$, the application,%
\begin{equation*}
U^{A}\longrightarrow \mathbb{A}^{n},\xi \longmapsto (\xi (x_{1}),\xi
(x_{2}),...,\xi (x_{n})),
\end{equation*}%
is a bijection from $U^{\mathbb{A}}$ into an open of $\mathbb{A}^{n}$. The
manifold $M^{\mathbb{A}}$ is modeled on $\mathbb{A}^{n}$, i.e $M^{\mathbb{A}%
} $ is a $\mathbb{A}$-manifold of dimension $n$ \cite{bo},\cite{shu},\cite%
{shu1}.

The set, $C^{\infty }(M^{\mathbb{A}},\mathbb{A})$, of smooth functions on $%
M^{\mathbb{A}}$ with values in $\mathbb{A}$ is a associative and commutative
algebra with unit, over $\mathbb{A}$ and for any $f\in C^{\infty }(M)$, the
application%
\begin{equation*}
f^{\mathbb{A}}:M^{\mathbb{A}}\longrightarrow \mathbb{A},\xi \longmapsto \xi
(f)
\end{equation*}%
is differentiable. Futher the application%
\begin{equation*}
C^{\infty }(M)\longrightarrow C^{\infty }(M^{\mathbb{A}},\mathbb{A}%
),f\longmapsto f^{\mathbb{A}},
\end{equation*}%
is a injective homomorphism of algebras and we have:%
\begin{equation*}
(f+g)^{\mathbb{A}}=f^{\mathbb{A}}+g^{\mathbb{A}};\text{ \ \ \ }(\lambda
\cdot f)^{\mathbb{A}}=\lambda \cdot f^{\mathbb{A}};\text{ \ \ \ }(f\cdot g)^{%
\mathbb{A}}=f^{\mathbb{A}}\cdot g^{\mathbb{A}}
\end{equation*}%
with $\lambda \in \mathbb{R}$, $f$ and $g$ belonging to $C^{\infty }(M)$.

\bigskip

In \cite{bo}, we have shown that there is an equivalence between the
following statements:

\begin{enumerate}
\item A vector field on $M^{\mathbb{A}}$ is a differentiable section of the
tangent bundle $(TM^{\mathbb{A}},\pi _{M^{\mathbb{A}}},M^{\mathbb{A}})$;

\item A vector field on $M^{\mathbb{A}}$ is a derivation of $C^{\infty }(M^{%
\mathbb{A}})$;

\item A vector field on $M^{\mathbb{A}}$ is a linear map $X:C^{\infty
}(M)\longrightarrow C^{\infty }(M^{\mathbb{A}},\mathbb{A})$ such that 
\begin{equation*}
X(f\cdot g)=X(f)\cdot g^{\mathbb{A}}+f^{\mathbb{A}}\cdot X(g)\text{, for any}%
\,f,g\in C^{\infty }(M)\text{;}
\end{equation*}
\end{enumerate}

and then, in \cite{nbo}, we have established that these statements are
equivalent to the assertion

\begin{enumerate}
\item[4.] A vector field on $M^{\mathbb{A}}$ is a derivation of $C^{\infty
}(M^{\mathbb{A}},\mathbb{A)}$ which is $\mathbb{A}$-linear.
\end{enumerate}

Thus the set $\mathfrak{X}(M^{\mathbb{A}})$ of vector fields on $M^{\mathbb{A%
}}$ is a $C^{\infty }(M^{\mathbb{A}},\mathbb{A})$-module.

Furthermore, we have \cite{nbo}

\begin{theorem}
The map%
\begin{equation*}
\mathfrak{X}(M^{\mathbb{A}})\times \mathfrak{X}(M^{\mathbb{A}%
})\longrightarrow \mathfrak{X}(M^{\mathbb{A}}),(X,Y)\longmapsto \lbrack
X,Y]=X\circ Y-Y\circ X
\end{equation*}%
is skew-symmetric $\mathbb{A}$-bilinear and defines a structure of $\mathbb{A%
}$-Lie algebra over $\mathfrak{X}(M^{\mathbb{A}})$. Moreover, for any $%
\varphi \in C^{\infty }(M^{\mathbb{A}},\mathbb{A})$, we have%
\begin{equation*}
\lbrack X,\varphi Y]=X\left( \varphi \right) Y+\varphi \lbrack X,Y]\text{.}
\end{equation*}
\end{theorem}

Recall that when $M$ is a smooth manifold, the basic algebra of $M$ is $%
C^{\infty }(M)$. Since $\mathfrak{X}(M^{\mathbb{A}})$ is a $C^{\infty }(M^{%
\mathbb{A}},\mathbb{A})$-module, and is a Lie algebra over $\mathbb{A}$, and
as $M^{\mathbb{A}}$ is a $\mathbb{A}$-manifold, this means that the basic
algebra of $M^{\mathbb{A}}$ is $C^{\infty }(M^{\mathbb{A}},\mathbb{A})$ and
not $C^{\infty }(M^{\mathbb{A}})$.

Thus, in all what follows, we denotes $\mathfrak{X}(M^{\mathbb{A}})$, the
set of $\mathbb{A}$-linear maps 
\begin{equation*}
X:C^{\infty }(M^{\mathbb{A}},\mathbb{A})\longrightarrow C^{\infty }(M^{%
\mathbb{A}},\mathbb{A})
\end{equation*}%
such that 
\begin{equation*}
X(\varphi \cdot \psi )=X(\varphi )\cdot \psi +\varphi \cdot X(\psi ),\quad 
\text{for any}\,\varphi ,\psi \in C^{\infty }(M^{\mathbb{A}},\mathbb{A})
\end{equation*}%
that is to say 
\begin{equation*}
\mathfrak{X}(M^{\mathbb{A}})=Der_{\mathbb{A}}[C^{\infty }(M^{\mathbb{A}},%
\mathbb{A})]\text{.}
\end{equation*}

\subsection{Differential form on $M^{\mathbb{A}}$}

Let $\mathcal{L}_{sks}^{p}\left[ \mathfrak{X}(M^{\mathbb{A}}),C^{\infty }(M^{%
\mathbb{A}},\mathbb{A})\right] =\Lambda _{\mathbb{A}}^{p}(M^{\mathbb{A}},%
\mathbb{A})$ be the $C^{\infty }(M^{\mathbb{A}},\mathbb{A})$-module of \
skew-symmetric $C^{\infty }(M^{\mathbb{A}},\mathbb{A})$-multilinear
applications for any $p\in \mathbb{N}$,%
\begin{equation*}
\omega :\underset{p\text{ times}}{\underbrace{\mathfrak{X}(M^{\mathbb{A}%
})\times \mathfrak{X}(M^{\mathbb{A}})\times ...\times \mathfrak{X}(M^{%
\mathbb{A}})}}\longrightarrow C^{\infty }(M^{\mathbb{A}},\mathbb{A})\text{.}
\end{equation*}

\bigskip\ We say that $\omega $ is a $A$-linear form of degree $p$ on $M^{%
\mathbb{A}}$.

We denote 
\begin{equation*}
\Lambda _{\mathbb{A}}(M^{\mathbb{A}})=\bigoplus\limits_{p=0}^{n}\Lambda _{%
\mathbb{A}}^{p}(M^{\mathbb{A}},\mathbb{A})
\end{equation*}%
and we have%
\begin{equation*}
\Lambda ^{0}(M^{\mathbb{A}},\mathbb{A})=C^{\infty }(M^{\mathbb{A}},\mathbb{A}%
)\text{.}
\end{equation*}

If $\omega $ is a differential form of degree $p$ on $M$, then there exists
an unique differential $A$-form of degree $p$ on $M^{A}$\ such that 
\begin{equation*}
\omega ^{\mathbb{A}}(\theta _{1}^{\mathbb{A}},\theta _{2}^{\mathbb{A}%
},...,\theta _{p}^{\mathbb{A}})=\left[ \omega (\theta _{1},\theta
_{2},...,\theta _{p})\right] ^{\mathbb{A}}
\end{equation*}%
for any vector fields $\theta _{1},\theta _{2},...,\theta _{p}\ \in 
\mathfrak{X}($ $M)$. \ We say that the differential $\mathbb{A}$-form $%
\omega ^{\mathbb{A}}$ is the prolongation to $M^{\mathbb{A}}$ of the
differential form $\omega $ \cite{mor}.

When 
\begin{equation*}
d:\Lambda (M)\longrightarrow \Lambda (M)
\end{equation*}%
is the exterior differential operator on $M$, we denote 
\begin{equation*}
d^{\mathbb{A}}:\Lambda _{\mathbb{A}}(M^{\mathbb{A}})\longrightarrow \Lambda
_{\mathbb{A}}(M^{\mathbb{A}})
\end{equation*}%
the cohomology operator associated to the representation 
\begin{equation*}
\mathfrak{X}(M^{\mathbb{A}})\longrightarrow Der_{\mathbb{A}}\left[ C^{\infty
}(M^{\mathbb{A}},\mathbb{A})\right] ,X\longmapsto X\text{.}
\end{equation*}

Thus, For $\eta \in \Lambda _{\mathbb{A}}^{p}(M^{\mathbb{A}})$ and for $%
X_{1},X_{2},...,X_{p+1}\in \mathfrak{X}(M^{\mathbb{A}})$, we have%
\begin{align*}
d^{\mathbb{A}}\eta (X_{1},X_{2},...,X_{p+1})&
=\sum_{i=1}^{p+1}(-1)^{i-1}X_{i}[\eta (X_{1},...,\widehat{X_{i}},...,X_{p+1}]
\\
& +\sum_{1\leq i<j\leq p+1}(-1)^{i+j}\eta (\left[ X_{i},X_{j}\right]
,X_{1},...,\widehat{X_{i}},...,\widehat{X_{j}},...,X_{p+1})
\end{align*}%
For $p=1,$ we have%
\begin{equation*}
d^{\mathbb{A}}\eta (X,Y)=X_{1}\left( \eta (X_{2})\right) -X_{2}\left( \eta
(X_{1})\right) -\eta (\left[ X_{1},X_{2}\right] )\text{.}
\end{equation*}

\bigskip

In the continuation, we define the concept of Lie-Rinehart structure on $M^{%
\mathbb{A}}$. Its anchor map is with values in the $C^{\infty }(M^{\mathbb{A}%
},\mathbb{A})$-module of first differential operators of $C^{\infty }(M^{%
\mathbb{A}},\mathbb{A})$ which are which are $\mathbb{A}$-linears.

\section{$\ $Lie-Rinehart structure on Weil bundle}

\subsection{Differential operator on $M^{\mathbb{A}}$}

A differential operator on $M^{\mathbb{A}}$, is a $%
\mathbb{R}
$-linear map%
\begin{equation*}
\delta :C^{\infty }(M^{\mathbb{A}})\longrightarrow \ C^{\infty }(M^{\mathbb{A%
}})\ 
\end{equation*}%
such that for any $\varphi ,\psi \in C^{\infty }(M^{\mathbb{A}})$,%
\begin{equation*}
\delta (\varphi \cdot \psi )=\delta (\varphi )\cdot \psi +\varphi \cdot
\delta (\psi )-\varphi \cdot \psi \cdot \delta (1_{C^{\infty }(M^{\mathbb{A}%
})})\text{.}
\end{equation*}%
We denote $\mathcal{D(}M^{\mathbb{A}})$ the $C^{\infty }(M^{\mathbb{A}})$%
-module of differential operators on $M^{\mathbb{A}}$. We have

\bigskip

\begin{theorem}
The following assertions are equivalent:

\begin{enumerate}
\item A differential operator on $M^{\mathbb{A}}$ is a differential operator
of $C^{\infty }(M^{\mathbb{A}})$;

\item A differential operator on $M^{\mathbb{A}}$ is a linear map $%
Y:C^{\infty }(M)\longrightarrow C^{\infty }(M^{\mathbb{A}},\mathbb{A})$ such
that 
\begin{equation*}
Y(f\cdot g)=Y(f)\cdot g^{\mathbb{A}}+f^{\mathbb{A}}\cdot Y(g)-f^{\mathbb{A}%
}\cdot g^{\mathbb{A}}\cdot Y(1_{C^{\infty }(M)}),\quad \text{for any}%
\,f,g\in C^{\infty }(M)\text{;}
\end{equation*}

\item A differential operator on $M^{\mathbb{A}}$ is a differential operator
of $C^{\infty }(M^{\mathbb{A}},\mathbb{A})$ which is $A$-linear.
\end{enumerate}
\end{theorem}

\bigskip

\begin{proof}
The implications $1.\Longleftrightarrow 2$ are shown in \cite{bos1}.

$2.\Longleftrightarrow 3.$

If $X:C^{\infty }(M)\longrightarrow C^{\infty }(M^{\mathbb{A}},\mathbb{A})$
is a differential operator, such that 
\begin{equation*}
X(f\cdot g)=X(f)\cdot g^{\mathbb{A}}+f^{\mathbb{A}}\cdot X(g)-f^{\mathbb{A}%
}\cdot g^{\mathbb{A}}\cdot X(1_{C^{\infty }(M)}),\quad \text{for any}%
\,f,g\in C^{\infty }(M)\text{,}
\end{equation*}%
the map%
\begin{equation*}
Y:C^{\infty }(M)\longrightarrow C^{\infty }(M^{\mathbb{A}},\mathbb{A}%
),f\longmapsto X(f)-f^{\mathbb{A}}\cdot X(1_{C^{\infty }(M)}),
\end{equation*}%
is a derivation. Given \cite{bo}, there exist a unique derivation%
\begin{equation*}
\widetilde{Y}:C^{\infty }(M^{\mathbb{A}},\mathbb{A})\longrightarrow
C^{\infty }(M^{\mathbb{A}},\mathbb{A})
\end{equation*}%
such that:

$1$- $\widetilde{Y}$ is $A$-linear;

$2$- $\widetilde{Y}(\sigma )\in C^{\infty }(M^{\mathbb{A}})$ for all $\sigma
\in C^{\infty }(M^{\mathbb{A}})$;

$3$- $\widetilde{Y}(f^{\mathbb{A}})=Y(f)$ for every $f\in C^{\infty }(M)$.

The map 
\begin{equation*}
\widetilde{X}:C^{\infty }(M^{\mathbb{A}},\mathbb{A})\longrightarrow
C^{\infty }(M^{\mathbb{A}},\mathbb{A}),\varphi \longmapsto \widetilde{Y}%
(\varphi )+\varphi \cdot X(1_{C^{\infty }(M)}),
\end{equation*}%
is such that :

$1$- $\widetilde{X}$ is $\mathbb{A}$-linear;

$2$- For any $\varphi \in C^{\infty }(M^{\mathbb{A}})$, $\left[ \widetilde{X}%
(\varphi )-\varphi \cdot X(1_{C^{\infty }(M)})\right] \in C^{\infty }(M^{%
\mathbb{A}})$;

$3$- $\widetilde{X}(f^{\mathbb{A}})=X(f)$ for every $f\in C^{\infty }(M)$.

Conversely, if $X:C^{\infty }(M^{\mathbb{A}},\mathbb{A})\longrightarrow
C^{\infty }(M^{\mathbb{A}},\mathbb{A})$ is a differential operator which is $%
\mathbb{A}$-linear, the map%
\begin{equation*}
Z:C^{\infty }(M^{\mathbb{A}},\mathbb{A})\longrightarrow C^{\infty }(M^{%
\mathbb{A}},\mathbb{A}),\varphi \longmapsto X(\varphi )-\varphi \cdot
X(1_{C^{\infty }(M)})\text{,}
\end{equation*}%
is a derivation.

Thus, the map 
\begin{equation*}
Y:C^{\infty }(M)\longrightarrow C^{\infty }(M^{\mathbb{A}},\mathbb{A}%
),f\longmapsto Z(f^{\mathbb{A}})\text{,}
\end{equation*}%
is a a differential operator such that%
\begin{equation*}
Y(f\cdot g)=Y(f)\cdot g^{\mathbb{A}}+f^{\mathbb{A}}\cdot Y(g)-f^{\mathbb{A}%
}\cdot g^{\mathbb{A}}\cdot Y(1_{C^{\infty }(M)}),\quad \text{for any}%
\,f,g\in C^{\infty }(M)\text{.}
\end{equation*}%
That ends the proof.
\end{proof}

We denotes $\mathcal{D}_{A}(M^{A})$, the $C^{\infty }(M^{\mathbb{A}},\mathbb{%
A})$-module of $\mathbb{A}$-linear differential operator%
\begin{equation*}
X:C^{\infty }(M^{\mathbb{A}},\mathbb{A})\longrightarrow C^{\infty }(M^{%
\mathbb{A}},\mathbb{A})
\end{equation*}%
such that 
\begin{equation*}
X(\varphi \cdot \psi )=X(\varphi )\cdot \psi +\varphi \cdot X(\psi )-\varphi
\cdot \psi \cdot X(1_{C^{\infty }(M^{\mathbb{A}},\mathbb{A})}),\quad \text{%
for any}\,\varphi ,\psi \in C^{\infty }(M^{\mathbb{A}},\mathbb{A})\text{.}
\end{equation*}

\bigskip

\begin{theorem}
The application%
\begin{equation*}
\left[ ,\right] :\mathcal{D}_{\mathbb{A}}(M^{\mathbb{A}})\times \mathcal{D}_{%
\mathbb{A}}(M^{\mathbb{A}})\longrightarrow \mathcal{D}_{\mathbb{A}}(M^{%
\mathbb{A}}),(X,Y)\longmapsto X\circ Y-Y\circ X,
\end{equation*}%
is skew-symmetric $\mathbb{A}$-bilinear and difine a structure of $\mathbb{A}
$-Lie algebra on $\mathcal{D}_{\mathbb{A}}(M^{\mathbb{A}})$.

Furthermore, for $\varphi \in C^{\infty }(M^{\mathbb{A}},\mathbb{A})$ and
for $X,Y\in \mathcal{D}_{\mathbb{A}}(M^{\mathbb{A}})$, we have 
\begin{equation*}
\left[ X,\varphi \cdot Y\right] =\left[ X(\varphi )-\varphi \cdot
X(1_{C^{\infty }(M^{A},A)})\right] \cdot Y+\varphi \cdot \left[ X,Y\right] 
\text{.}
\end{equation*}
\end{theorem}

\bigskip

\begin{proof}
For $X,Y\in \mathcal{D}_{\mathbb{A}}(M^{\mathbb{A}})$, the map 
\begin{equation*}
\left[ ,\right] :\mathcal{D}_{\mathbb{A}}(M^{\mathbb{A}})\times \mathcal{D}_{%
\mathbb{A}}(M^{\mathbb{A}})\longrightarrow \mathcal{D}_{\mathbb{A}}(M^{%
\mathbb{A}}),(X,Y)\longmapsto X\circ Y-Y\circ X,
\end{equation*}%
is manifestly skew-symmetric $%
\mathbb{R}
$-bilinear. For $\varphi $ and $\psi $ belonging to $C^{\infty }(M^{\mathbb{A%
}},\mathbb{A})$, we verify that 
\begin{equation*}
\left[ X,Y\right] (\varphi \cdot \psi )=\left[ X,Y\right] (\varphi )\cdot
\psi +\varphi \cdot \left[ X,Y\right] (\psi )-\varphi \cdot \psi \cdot \left[
X,Y\right] (1_{C^{\infty }(M^{\mathbb{A}},\mathbb{A})})
\end{equation*}%
Thus $\left[ X,Y\right] \in \mathcal{D}_{\mathbb{A}}(M^{\mathbb{A}})$.

For $a\in \mathbb{A}$ and for $\varphi $ $\in $ $C^{\infty }(M^{\mathbb{A}},%
\mathbb{A})$, we have: 
\begin{align*}
\left[ X,a\cdot Y\right] (\varphi )& =X\left[ a\cdot Y(\varphi )\right]
-\left( a\cdot Y\right) \left[ X(\varphi \ )\right]  \\
& =a\cdot X\left[ Y(\varphi \ )\right] -a\cdot Y\left[ X(\varphi \ )\right] 
\\
& =a\cdot \left[ X,Y\right] (\varphi \ )\text{.}
\end{align*}%
Thus%
\begin{equation*}
\left[ X,a\cdot Y\right] =a\cdot \left[ X,Y\right] 
\end{equation*}%
i.e the map%
\begin{equation*}
\left[ ,\right] :\mathcal{D}_{\mathbb{A}}(M^{\mathbb{A}})\times \mathcal{D}_{%
\mathbb{A}}(M^{\mathbb{A}})\longrightarrow \mathcal{D}_{\mathbb{A}}(M^{%
\mathbb{A}}),(X,Y)\longmapsto X\circ Y-Y\circ X,
\end{equation*}%
is $\mathbb{A}$-bilinear and for $X,Y\in \mathcal{D}_{\mathbb{A}}(M^{\mathbb{%
A}})$ and for $\varphi \in $ $C^{\infty }(M^{\mathbb{A}},\mathbb{A})$, the
map 
\begin{equation*}
C^{\infty }(M^{\mathbb{A}},\mathbb{A})\longrightarrow C^{\infty }(M^{\mathbb{%
A}},\mathbb{A}),\varphi \longmapsto \lbrack X,Y](\varphi )=X\left[ Y(\varphi
)\right] -Y\left[ X(\varphi )\right] 
\end{equation*}%
is $\mathbb{A}$-linear.

Futhermore, for $\varphi \in C^{\infty }(M^{\mathbb{A}},\mathbb{A})$ and for 
$\psi \in $ $C^{\infty }(M)$, we have%
\begin{align*}
\left[ X,\varphi \cdot Y\right] (\psi \ )& =X\left[ \varphi \cdot Y(\psi \ )%
\right] -\left( \varphi \cdot Y\right) \left[ X(\psi \ )\right]  \\
& =X(\varphi )\cdot Y(\psi \ )+\varphi \cdot X\left[ Y(f\psi )\right]
-\varphi \cdot Y(\psi \ )\cdot X(1_{C^{\infty }(M^{\mathbb{A}},\mathbb{A}%
)})-\varphi \cdot Y\left[ X(\psi \ )\right]  \\
& =X(\varphi )\cdot Y(\psi \ )-\varphi \cdot Y(\psi \ )\cdot X(1_{C^{\infty
}(M^{\mathbb{A}},\mathbb{A})})+\varphi \cdot X\left[ Y(\psi \ )\right]
-\varphi \cdot Y\left[ X(\psi \ )\right]  \\
& =X(\varphi )\cdot Y(f\psi )-\varphi \cdot Y(\psi \ )\cdot X(1_{C^{\infty
}(M^{\mathbb{A}},\mathbb{A})})+\varphi \cdot \left[ X,Y\right] (\psi \ ) \\
& =\left[ X(\varphi )-\varphi \cdot X(1_{C^{\infty }(M^{\mathbb{A}},\mathbb{A%
})})\right] \cdot Y(\psi \ )+\varphi \cdot \left[ X,Y\right] (\psi )\text{.}
\end{align*}%
Thus%
\begin{equation*}
\left[ X,\varphi \cdot Y\right] =\left[ X(\varphi )-\varphi \cdot
X(1_{C^{\infty }(M^{\mathbb{A}},\mathbb{A})})\right] \cdot Y+\varphi \cdot %
\left[ X,Y\right] 
\end{equation*}%
That ends the demonstration.
\end{proof}

\bigskip

\subsection{Lie-Rinehart structure on Weil bundle}

\bigskip Let $R$ be a unitary commutative algebra over a field $\mathbb{K}$
of characteristic zero with unit $1_{R}$ and let $\mathcal{G}$ be a $R$%
-module with a structure of $\mathbb{K}$-Lie algebra.

A Lie-Rinehart algebra structure on $\mathcal{G}$ is given by a morphism of $%
R$-modules and of $\mathbb{K}$-Lie algebras%
\begin{equation*}
\rho :\mathcal{G}\longrightarrow \mathcal{D}(R)
\end{equation*}%
called anchor map, such that for all $x$ and $y$ in $\mathcal{G}$ and for
all $a\in R$,%
\begin{equation*}
\lbrack x,ay]=\left( \rho (x)(a)-a\cdot \rho (x)(1_{R})\right) \cdot
y+a\cdot \lbrack x,y].
\end{equation*}

\bigskip Thus, in the case of Weil bundles, Where $R$ is the algebra $%
C^{\infty }(M^{\mathbb{A}},\mathbb{A})$ and $\mathcal{G}$ $=\mathfrak{X}(M^{%
\mathbb{A}})$,

A Lie-Rinehart structure on Weil bundle $M^{\mathbb{A}}$ is given by a
morphism of $C^{\infty }(M^{\mathbb{A}},\mathbb{A})$-modules and of Lie
algabras over $\mathbb{A}$,

\begin{equation*}
\rho :\mathfrak{X}(M^{A})\longrightarrow \mathcal{D}_{A}
\end{equation*}%
which is $A$-linear and such that for $X,Y\in \mathfrak{X}(M^{\mathbb{A}})$
and for any $\varphi \in C^{\infty }(M^{\mathbb{A}},\mathbb{A})$, we get 
\begin{equation*}
\lbrack X,\varphi \cdot Y]=\left[ \rho (X)(\varphi )-\varphi \cdot \rho
(X)(1_{C^{\infty }(M^{\mathbb{A}},\mathbb{A})})\right] \ \cdot Y+\varphi
\cdot \lbrack X,Y]\text{.}
\end{equation*}

We say that the pair $(\mathfrak{X}(M^{\mathbb{A}}),\rho )$ is a
Lie-Rinehart structure on $M^{\mathbb{A}}$.

The anchor $\rho $ is a represention from $\mathfrak{X}(M^{\mathbb{A}})\ $%
into $C^{\infty }(M^{\mathbb{A}},\mathbb{A})$.

\bigskip

For any $p\in \mathbb{N}$,%
\begin{equation}
\Lambda ^{p}(M^{\mathbb{A}},\rho )=\mathcal{L}_{sks}^{p}[\mathfrak{X}(M^{%
\mathbb{A}}),C^{\infty }(M^{\mathbb{A}},\mathbb{A})]  \notag
\end{equation}%
denotes the $C^{\infty }(M^{\mathbb{A}},\mathbb{A})$-module of
skew-symmetric multilinear forms of degree $p$ from $\mathfrak{X}(M^{\mathbb{%
A}})$ into $C^{\infty }(M^{\mathbb{A}},\mathbb{A})$. We have 
\begin{equation}
\Lambda ^{0}(M^{\mathbb{A}},\rho )=C^{\infty }(M^{\mathbb{A}},\mathbb{A}). 
\notag
\end{equation}%
We denote%
\begin{equation}
\Lambda (M^{\mathbb{A}},\rho )=\bigoplus_{p=0}^{n}\Lambda ^{p}(M^{\mathbb{A}%
},\rho )\text{.}  \notag
\end{equation}

\bigskip

\begin{theorem}
\bigskip Let $(\mathfrak{X}(M^{\mathbb{A}}),\rho )$ be a Lie-Rinehart
structure on $M^{\mathbb{A}}$. There exists a $\mathbb{A}$-diff\'{e}rential
form of degrre $1$ on $M^{\mathbb{A}}$,%
\begin{equation*}
\alpha :\mathfrak{X}(M^{\mathbb{A}})\longrightarrow C^{\infty }(M^{\mathbb{A}%
},\mathbb{A})
\end{equation*}%
such that 
\begin{equation*}
\rho (X)(\varphi )=X(\varphi )+\varphi \cdot \alpha (X)\text{.}
\end{equation*}
\end{theorem}

\begin{proof}
We denote,%
\begin{equation*}
d_{\rho }:\Lambda (M^{\mathbb{A}},\rho )\longrightarrow \Lambda (M^{\mathbb{A%
}},\rho )
\end{equation*}%
the differential operator of degre $+1$ and of square $0$ associed to the
representation $\rho $.

For $X,Y\in \mathfrak{X}(M^{\mathbb{A}})$, we have (Theorem $1$) 
\begin{equation*}
\lbrack X,\varphi Y]=X(\varphi )\cdot Y+\varphi \cdot \lbrack X,Y]\text{.}
\end{equation*}%
Since $(\mathfrak{X}(M^{\mathbb{A}}),\rho )$ is a Lie-Rinehart structure on $%
M^{\mathbb{A}}$, then for $X,Y\in \mathfrak{X}(M^{\mathbb{A}})$ and for any $%
\varphi \in C^{\infty }(M^{\mathbb{A}},\mathbb{A})$,%
\begin{equation*}
\lbrack X,\varphi \cdot Y]=\left[ \rho (X)(\varphi )-\varphi \cdot \rho
(X)(1_{C^{\infty }(M^{\mathbb{A}},\mathbb{A})})\right] \ \cdot Y+\varphi
\cdot \lbrack X,Y]\text{.}
\end{equation*}

Thus,%
\begin{equation*}
X(\varphi )=\rho (X)(\varphi )-\varphi \cdot \rho (X)(1_{C^{\infty }(M^{%
\mathbb{A}},\mathbb{A})})\ 
\end{equation*}%
i.e%
\begin{equation*}
\rho (X)(\varphi )=X(\varphi )+\varphi \cdot \rho (X)(1_{C^{\infty }(M^{%
\mathbb{A}},\mathbb{A})})\text{.}
\end{equation*}

Then, the map%
\begin{equation*}
\alpha :\mathfrak{X}(M^{\mathbb{A}})\longrightarrow C^{\infty }(M^{\mathbb{A}%
},\mathbb{A}),X\longmapsto \alpha (X)=\rho (X)(1_{C^{\infty }(M^{\mathbb{A}},%
\mathbb{A})})
\end{equation*}%
is a $\mathbb{A}$-diff\'{e}rential form of degree $1$ on $M^{\mathbb{A}}$
that answers the question, i.e 
\begin{equation*}
\rho (X)(\varphi )=X(\varphi )+\varphi \cdot \alpha (X).
\end{equation*}
\end{proof}

\bigskip

\begin{proposition}
If $d^{\mathbb{A}}$ denote the differential operator of degre $+1$ and of
square $0$ associed to the representation $\mathfrak{X}(M^{\mathbb{A}%
})\longrightarrow Der_{A}(C^{\infty }(M^{\mathbb{A}},\mathbb{A}%
)),X\longmapsto X$, then the $A$-diff\'{e}rential $1$- form $\alpha $ is $d^{%
\mathbb{A}}$-closed i.e $d^{\mathbb{A}}\alpha =0$.
\end{proposition}

\bigskip

\begin{proof}
For any $X,Y\in \mathfrak{X}(M^{\mathbb{A}})$ and\ for any $\varphi \in
C^{\infty }(M^{\mathbb{A}},\mathbb{A})$,%
\begin{eqnarray*}
([\rho (X),\rho (Y)]-\rho \lbrack X,Y])(\varphi ) &=&\rho (X)[\rho
(Y)(\varphi )]-\rho (Y)[\rho (X)(\varphi )]-\rho \lbrack X,Y])(\varphi ) \\
&=&\rho (X)[Y(\varphi )+\varphi \cdot \alpha (Y)]-\rho (Y)[X(\varphi
)+\varphi \cdot \alpha (X)] \\
&&-[X,Y](\varphi )-\varphi \cdot \alpha ([X,Y]) \\
&=&X[Y(\varphi )+\varphi \cdot \alpha (Y)]+[Y(\varphi )+\varphi \cdot \alpha
(Y)]\cdot \alpha (X) \\
&&-Y[X(\varphi )+\varphi \cdot \alpha (X)]-[X(\varphi )+\varphi \cdot \alpha
(X)]\cdot \alpha (Y) \\
&&-[X,Y](\varphi )-\varphi \cdot \alpha ([X,Y])\text{.}
\end{eqnarray*}%
i.e%
\begin{eqnarray*}
([\rho (X),\rho (Y)]-\rho \lbrack X,Y])(\varphi ) &=&X(Y(\varphi
))+X(\varphi \cdot \alpha (Y))+Y(\varphi )\cdot \alpha (X) \\
&&+\varphi \cdot \alpha (Y)\ \cdot \alpha (X)-Y(X(\varphi ))-Y(\varphi \cdot
\alpha (X))-X(\varphi )\cdot \alpha (Y) \\
&&-\varphi \cdot \alpha (X)\cdot \alpha (Y)-[X,Y](\varphi )-\varphi \cdot
\alpha ([X,Y]) \\
&=&X(Y(\varphi ))+X(\varphi )\cdot \alpha (Y)+\varphi \cdot X(\alpha
(Y))+Y(\varphi )\cdot \alpha (X) \\
&&+\varphi \cdot \alpha (Y)\ \cdot \alpha (X)-Y(X(\varphi ))-Y(\varphi
)\cdot \alpha (X)-\varphi \cdot Y(\alpha (X)) \\
&&-X(\varphi )\cdot \alpha (Y)-\varphi \cdot \alpha (X)\cdot \alpha
(Y)-[X,Y](\varphi )-\varphi \cdot \alpha ([X,Y]) \\
&=&\varphi \cdot (X(\ \alpha (Y))-Y(\ \alpha (X))-\alpha ([X,Y]) \\
&=&\ \varphi \cdot \lbrack d^{\mathbb{A}}\alpha ](X,Y)\text{.}
\end{eqnarray*}

Since $\varphi $ is morphism of $A$-Lie algabras, we have then $\ \rho
\lbrack X,Y])=[\rho (X),\rho (Y)]\ $\ if only if $[d^{\mathbb{A}}\alpha
](X,Y)=0.$

That ends the proof.
\end{proof}

The $1$-form $\alpha $ is the canonical form of the structure of
Lie-Rinehart algebra $(\mathfrak{X}(M^{\mathbb{A}}),\rho )$.

\begin{corollary}
Let $\alpha $ be a differential $\mathbb{A}$-form of degree $1$ on $M^{%
\mathbb{A}}$ and let a representation 
\begin{equation*}
\rho _{\alpha }:\mathfrak{X}(M^{\mathbb{A}})\longrightarrow \mathcal{D}_{%
\mathbb{A}}
\end{equation*}%
such that 
\begin{equation*}
\rho _{\alpha }(X)(\varphi )=X(\varphi )+\varphi \cdot \alpha (X)
\end{equation*}%
for any $\varphi \in C^{\infty }(M^{\mathbb{A}},\mathbb{A})$. The pair $(%
\mathfrak{X}(M^{\mathbb{A}}),\rho _{\alpha })$ is a $\mathbb{A}$%
-Lie-Rinehart structure if and only if $d^{\mathbb{A}}\alpha =0$.
\end{corollary}

\bigskip\


\begin{thebibliography}{99}
\bibitem{bo} B.G.R. Bossoto, E. Okassa, \textit{Champs de vecteurs et Formes
diff\'{e}rentielles sur une vari\'{e}t\'{e} des points proches}, Archivum
Math. (Brno), 44 (2008) 159--171.

\bibitem{bos1} B.G.R. Bossoto, \textit{Structures de Jacobi sur une vari\'{e}%
t\'{e} des points proches}, Matematicki Vesnik, 62, 2 (2010), 155-167.

\bibitem{ko1} I. Kolar, \textit{Affine structure on Weil bundles,} Nagoya
Math. J. 158 (2000), 99-06.

\bibitem{kms} P. Kol\'{a}r, P. W. Michor and J. Slovak, \textit{Natural
Operations in Differential Geometry}, Springer-Verlag, Berlin, 1993.

\bibitem{mor} A. Morimoto, \textit{Prolongation of connections to bundles of
infinitely near points}, J.Diff.Geom., t.11, 1976, 479-498.

\bibitem{nbo} B.V. Nkou,~ B.G.R. Bossoto, ~ E. Okassa, \textit{New
properties of prolongations of Linear connections on Weil bundles}, to
appear.

\bibitem{ok1} E. Okassa, \textit{Prolongement des champs de vecteurs \`{a}
des vari\'{e}t\'{e}s des points proches}, Annales Facult\'{e} des Sciences
de Toulouse, Vol. VIII, 3, 1986-1987, 349-366.

\bibitem{ok2} E. Okassa, \textit{Alg\`{e}bres de Jacobi et Alg\`{e}bres de
Lie-Rinehart-Jacobi}, J.of pure and applied algebra, vol. 208, 3 (2007),
1071-1089.

\bibitem{shu} V.\ V.\ Shurygin, \textit{Smooth manifolds over local
algebras, }J. Math. Sciences, Vol. 108, 3 (2002), 249-294.

\bibitem{shu1} V.\ V.\ Shurygin, \textit{Some aspects of the theory of
manifolds over algebras and Weil bundles, }J. Math. Sciences, vol. 169, 3
(2010), 315-341\textit{.}

\bibitem{wei} A. Weil, \textit{Th\'{e}orie des points proches sur les vari%
\'{e}t\'{e}s diff\'{e}rentiables}, Colloque G\'{e}om. Diff. Strasbourg,
1953, 111-117.
\end{thebibliography}
\end{document}